\documentclass[11pt,a4paper]{amsart}

\usepackage{amssymb,amsfonts,mathrsfs,enumitem,amsmath,amsthm,bbold}
\usepackage[mathscr]{euscript}
\usepackage[all]{xy}

\usepackage{color,fourier}
\definecolor{Chocolat}{rgb}{0.36, 0.2, 0.09}
\definecolor{BleuTresFonce}{rgb}{0.215, 0.215, 0.36}
\definecolor{EgyptianBlue}{rgb}{0.06, 0.2, 0.65}

\usepackage{forest,adjustbox}

\usepackage[T2A,T1]{fontenc}
\usepackage[utf8]{inputenc}
\DeclareTextSymbolDefault{\cyrsh}{T2A}

\usepackage[colorlinks,final,hyperindex]{hyperref}

\catcode`\@=11

\catcode`\@=13

\newtheorem*{itheorem}{Theorem}
\newtheorem{theorem}{Theorem}
\newtheorem{lemma}{Lemma}
\newtheorem{corollary}{Corollary}
\newtheorem{proposition}{Proposition}

\theoremstyle{definition}

\DeclareMathAlphabet{\pazocal}{OMS}{zplm}{m}{n}

\def\calM{\pazocal{M}}

\def\calP{\pazocal{P}}
\def\calQ{\pazocal{Q}}
\def\calR{\pazocal{R}}

\def\calX{\pazocal{X}}
\def\calY{\pazocal{Y}}
\def\calZ{\pazocal{Z}}

\DeclareMathAlphabet{\mathbbold}{U}{bbold}{m}{n}

\def\kk{\mathbbold{k}}

\DeclareMathOperator{\id}{id}

\DeclareMathOperator{\Hom}{Hom}

\DeclareMathOperator{\PreLie}{\textsl{PreLie}}

\DeclareMathOperator{\NAP}{\textsl{NAP}}
\DeclareMathOperator{\Perm}{\textsl{Perm}}
\DeclareMathOperator{\Lie}{\textsl{Lie}}
\DeclareMathOperator{\Ass}{\textsl{Ass}}
\DeclareMathOperator{\Com}{\textsl{Com}}
\DeclareMathOperator{\uAss}{\textsl{uAss}}
\DeclareMathOperator{\uCom}{\textsl{uCom}}
\DeclareMathOperator{\RT}{\textsl{RT}}

\DeclareMathOperator{\FM}{\textsl{FMan}}
\DeclareMathOperator{\Poisson}{\textsl{Poisson}}

\newcommand{\ac}{\scriptstyle \text{\,\rm !`}}

\begin{document}

\title{Three Schur functors related to pre-Lie algebras}

\author{Vladimir Dotsenko}
\address{Institut de Recherche Math\'ematique Avanc\'ee, UMR 7501, Universit\'e de Strasbourg et CNRS, 7 rue Ren\'e-Descartes, 67000 Strasbourg CEDEX, France}
\email{vdotsenko@unistra.fr}

\author{Ois\'in Flynn--Connolly}
\address{Universit\'e Sorbonne Paris Nord, Laboratoire de G\'eom\'etrie, Analyse et Applications, LAGA, CNRS, UMR 7539, F-93430, Villetaneuse, France}
\email{flynn-connolly@math.univ-paris13.fr}


\begin{abstract}
We give explicit combinatorial descriptions of three Schur functors arising in the theory of pre-Lie algebras. The first of them leads to a functorial description of the underlying vector space of the universal enveloping pre-Lie algebra of a given Lie algebra, strengthening the PBW theorem of Segal. The two other Schur functors provide functorial descriptions of the underlying vector spaces of the universal multiplicative enveloping algebra and of the module of K\"ahler differentials of a given pre-Lie algebra. An important consequence of such descriptions is an interpretation of the cohomology of a pre-Lie algebra with coefficients in a module as a derived functor for the category of modules over the universal multiplicative enveloping algebra. 
\end{abstract}

\maketitle

\section{Introduction}

In this paper, we establish several new results on pre-Lie algebras. The algebraic structure of a pre-Lie algebra, or a right-symmetric algebra, is formally defined as a vector space $V$ equipped with a binary operation $\triangleleft$ satisfying the identity
 \[
(a_1\triangleleft a_2)\triangleleft a_3 - a_1\triangleleft (a_2\triangleleft a_3)=(a_1\triangleleft a_3)\triangleleft a_2 - a_1\triangleleft (a_3\triangleleft a_2) .
 \]
Any pre-Lie algebra is a Lie algebra with respect to the commutator $[a_1,a_2]:=a_1\triangleleft a_2-a_2\triangleleft a_1$, so one may consider universal enveloping pre-Lie algebras of Lie algebras. They have been first studied by Segal in \cite{Se1994} who was motivated by geometric questions such as the classification problem of left invariant locally flat affine structures on Lie groups; the relevance of pre-Lie algebras for such purposes was discovered by Vinberg \cite{Vin1963}. Segal found a certain basis of nonassociative words for the universal enveloping algebra; his answer does not depend on the Lie algebra structure on $\mathfrak{g}$, and therefore one can say that an analogue of the classical Poincar\'e--Birkhoff--Witt (PBW) theorem holds for universal enveloping pre-Lie algebras of Lie algebras, along the lines of the general approach to PBW theorems proposed by Mikhalev and Shestakov in~\cite{MS2014}. Moreover, Bolgar proved in \cite{MR1384463} that for the natural filtration on $U^{\PreLie}(\mathfrak{g})$, the associated graded pre-Lie algebra is isomorphic to the universal enveloping pre-Lie algebra of the abelian Lie algebra with the same underlying vector space as $\mathfrak{g}$, which brings universal enveloping pre-Lie algebras closer to the context of the category theoretical approach to PBW theorems proposed by the first author and Tamaroff in \cite{DT2019} who proved that a functorial PBW theorem holds in this case using a previous result \cite{Do2019} of the first author. Neither of the abovementioned results, however, leads to an explicit description of the Schur functor of $U^{\PreLie}(\mathfrak{g})$. 
The first part of the paper is concerned with furnishing such a description. 

The classical Poincar\'e--Birkhoff--Witt theorem in the associative case is centered around the space of symmetric tensors $S(\mathfrak{g})$. It is well known that monomials in pre-Lie are rooted trees whose vertices are labelled by generators: implicit in the works of Cayley \cite{Cayley1857}, it has been proved rigorously in \cite{CL2001
}. The main result of this paper gives an explicit description of the Schur functor of the universal enveloping pre-Lie algebra of a Lie algebra in a way that is reminiscent of the classical Poincar\'e--Birkhoff--Witt theorem and at the same time highlights the appealing combinatorics of rooted trees. 
\begin{itheorem}[{Th.~\ref{th:LiePreLie}}]
Let $\mathfrak{g}$ be a Lie algebra. There is a vector space isomorphism
 \[
U^{\PreLie}(\mathfrak{g})\cong \RT_{\ne 1}(S(\mathfrak{g})),
 \]
where $\RT_{\ne 1}$ is the linear species of rooted trees for which no vertex has exactly one child; moreover, these isomorphisms can be chosen in a way that is natural with respect to Lie algebra morphisms. 
\end{itheorem}

Our proof uses a modification of the homological criterion of freeness \cite[Prop.~4.1]{DT2019} that allows us to utilise the underlying rooted tree structure, bringing in standard techniques for working with graph complexes \cite{MR3348138}. The species we found has been previously studied by graph theorists who established that it can also be viewed as the species of ``labelled connected $P_4$-free chordal graphs'', see \cite{MR2031002} and the entry A058863 in \cite{oeis} for details. 

This result has some immediate applications, of which we give three: a proof of a similar result for the operad of $F$-manifold algebras \cite{Do2019}, a formula for the permutative bar homology of an associative commutative algebra, and a hint that can hopefully be used to construct a conjectural good triple of operads \cite{Loday2008} $(\calX^c,\PreLie,\Lie)$ that would allow one to prove a Milnor--Moore type theorem for universal enveloping pre-Lie algebras. 

The second part of the paper is concerned with similar results concerning universal multiplicative enveloping algebras and modules of K\"ahler differentials of pre-Lie algebras. In general, for an operad $\calP$ and a $\calP$-algebra $L$, the universal multiplicative enveloping algebra, is the associative algebra $U_{\calP}(V)$ whose category of left modules is equivalent to the category of operadic $L$-modules. For a pre-Lie algebra $L$, this object was studied in \cite{KU2004} by means of noncommutative Gr\"obner bases, and its monomial basis was constructed. In this paper, we use the approach of~\cite{Kh2019} to prove the following functorial version of the description of  universal multiplicative enveloping algebras of pre-Lie algebras:
\begin{itheorem}[{Th.~\ref{th:AsPreLie}}]
Let $L$ be a pre-Lie algebra. There is a vector space isomorphism
 \[
U_{\PreLie}(L)\cong T(L)\otimes S(L).
 \]
Moreover, these isomorphisms can be chosen in a way that is natural with respect to pre-Lie algebra morphisms. 
\end{itheorem}

In \cite{KU2004}, a basis of the module of K\"ahler differentials $\Omega^1_{\PreLie}(L)$ of a given pre-Lie algebra $L$ was also constructed. We prove a new functorial PBW theorem for modules of K\"ahler differentials of algebras over operads and apply it in the case of pre-Lie algebras, obtaining the following result.
\begin{itheorem}[{Th.~\ref{th:KahlerPreLie}}]
Let $L$ be a pre-Lie algebra. There is a vector space isomorphism
 \[
\Omega^1_{\PreLie}(L)\cong \overline{T}(L),
 \]
where $\overline{T}(L)$ is the augmentation ideal of the tensor algebra. Moreover, these isomorphisms can be chosen in a way that is natural with respect to pre-Lie algebra morphisms. 
\end{itheorem}

We then apply these two results to give a new recipe for calculation of the cohomology of a pre-Lie algebra $L$ with coefficients in a module $E$, proving that
 \[
H_{\PreLie}^\bullet(L,E)\cong{\mathrm{Ext}}^\bullet_{U_{\PreLie}(L)}(\Omega^1_{\calP}(L),E).
 \] 

This is a short note, and we do not intend to overload it with excessive recollections. We refer the reader to \cite{LV2012} for relevant information on symmetric operads, Koszul duality, and operadic twisting cochains, and to \cite{BLL1998} for information on combinatorics of species. In particular, following \cite{LV2012}, we refer to analytic endofunctors on the category of vector spaces as Schur functors, and use the notation $\kappa$ for the twisting morphism between an operad and its Koszul dual cooperad. All operads in this paper are defined over a field $\kk$ of characteristic zero, and are assumed weight graded and connected; we denote by $\mathbb{I}$ the trivial operad (linear span of the unit). All chain complexes are homologically graded (with the differential of degree $-1$); we use the notation $s$ to implement shifts of homological degree. When writing down elements of operads, we use Latin letters as placeholders; when working with algebras over operads that carry nontrivial homological degrees, there are extra signs which arise from applying operations to arguments via the usual Koszul sign rule.

\section{Functorial PBW theorems}

We begin with giving a brief recollection of the recent results on functorial Poincar\'e--Birkhoff--Witt theorems, and proving a similar new result on K\"ahler differentials.  

\subsection{Universal enveloping \texorpdfstring{$\calQ$}{Q}-algebras of \texorpdfstring{$\calP$}{P}-algebras}
The universal enveloping algebras of the first kind are defined whenever one is given a morphism of operads 
\[\phi\colon\calP\rightarrow\calQ.\] 
Such a morphism leads to a natural functor $\phi^*$ from the category of $\calQ$-algebras to the category of $\calP$-algebras (pullback of the structure). It is well known \cite{MR0233864,MR0244341} that this functor admits a left adjoint $\phi_!$ which is sometimes called the universal enveloping $\calQ$-algebra of~$L$. Moreover, we shall use the observation \cite{Re1996} that this left adjoint can be computed via the relative composite product formula \[\phi_!(L)=\calQ\circ_\calP L,\] where  $L$ in the latter formula is regarded as a ``constant Schur functor'' (a symmetric sequence supported at arity zero) given by the underlying vector space of the $\calP$-algebra $L$.  

In joint work with Tamaroff \cite{DT2019}, the first author gave a categorical definition of what it means for the datum $(\calP,\calQ,\phi)$ to have the PBW property: by definition, one requires that there exists a Schur functor~$\calX$ such that the underlying object of the universal enveloping $\calQ$-algebra of any $\calP$-algebra~$V$ is isomorphic to the evaluation of $\calX$ on $V$ (given by~$\calX\circ_{\mathbb{I}}V$) naturally with respect to $\calP$-algebra morphisms. According to \cite[Th.~3.1]{DT2019}, the datum $(\calP,\calQ,\phi)$ has the PBW property if and only if the right $\calP$-module action on $\calQ$ via $\phi$ is free; in this case, the Schur functor $\calX$ that generates that right module satisfies the above condition:
 \[\phi_!(V)\cong\calX\circ_{\mathbb{I}}V\] naturally with respect to $\calP$-algebra morphisms.

\subsection{Universal multiplicative enveloping algebras of \texorpdfstring{$\calP$}{P}-algebras}
The universal enveloping algebra of the second kind is defined for any algebra $V$ over an operad~$\calP$: it is an associative algebra $U_{\calP}(V)$ whose category of left modules is equivalent to the category of $V$-modules defined by means of operad theory; in the ``pre-operad'' literature, this object is often referred to as the universal multiplicative enveloping algebra. Let us briefly summarise the relevant background information here, following~\cite{Fr2009,Kh2019}. First, one considers a particular type of $\{1,2\}$-coloured operads, namely those whose structure operations can either have all inputs and the output of colour $1$ or all inputs but one of colour $1$ and the remaining input as well as the output of colour $2$. Such an operad is a pair $(\calQ,\calR)$, where $\calQ$ is a usual operad, and $\calR$ is a right $\calQ$-module in the category of twisted associative algebras, or in other words, a $\Ass$-$\calQ$-bimodule. 

For a usual operad $\calP$, one can consider the derivative $\partial(\calP)$ defined by \[\partial(\calP)(I):=\calP(I\sqcup\{\star\}),\] and define a $\{1,2\}$-coloured Schur functor $(\calP,\partial(\calP))$, where by definition the input $\star$ and the output of $\partial(\calP)$ are of colour $2$. This Schur functor has a $\{1,2\}$-coloured operad structure arising from the operad structure on $\calP$, and algebras $(V,M)$ over this coloured operad are precisely a $\calP$-algebra $V$ and a $V$-module~$M$. As we assume all operads connected, the augmentation $\calP\to\mathbb{I}$ of the operad $\calP$ may be used to make the pair $(\calP,\partial(\mathbb{I}))$ a $\{1,2\}$-coloured operad. The unit $\eta\colon\mathbb{I}\to\calP$ of the operad $\calP$ gives rise to a morphism of two-coloured operads \[\psi\colon (\calP,\partial(\mathbb{I}))\to (\calP,\partial(\calP)),\] and if one denotes by $\kk$ the trivial $V$-module (that is, the module on which all the operations of the augmentation ideal of $\calP$ vanish), we have 
 \[
(V,U_{\calP}(V))\cong (\calP,\partial(\calP))_{(\calP,\partial(\mathbb{I}))}(V,\kk). 
 \]
From this observation and \cite[Th.~3.1]{DT2019}, it immediately follows that a functorial PBW type theorem for universal multiplicative enveloping algebras holds if and only if the Schur functor $\partial(\calP)$ is a free right $\calP$-module; in this case, the Schur functor $\calY$ that generates that right module satisfies $U_{\calP}(A)\cong\calY\circ_{\mathbb{I}}A$ naturally with respect to $\calP$-algebra morphisms.

\subsection{K\"ahler differentials of \texorpdfstring{$\calP$}{P}-algebras}

Mimicking the classical definition from commutative algebra, one may define the $U_{\calP}(L)$-module of K\"ahler differentials $\Omega^1_\calP(L)$ for any given operad $\calP$ and any given $\calP$-algebra $L$ \cite[Sec.~4.4]{Fr2009}. The intrinsic definition states that $\Omega^1_\calP(L)$ is the $U_{\calP}(L)$-module that represents the functor of derivations $\mathrm{Der}(L,E)$ with values in a $U_{\calP}(V)$-module $E$. It is known that this module can be constructed in a very explicit way: it is spanned by formal expressions $p(a_1,\ldots,da_i,\ldots,a_m)$, where $p\in\calP(m)$, and $a_1, \ldots, a_m\in L$, modulo the relations
\begin{multline*}
p(a_1,\ldots,q(a_i,\ldots,a_{i+n-1}),\ldots, da_j,\ldots, a_{m+n-1})=\\
(p\circ_i q)(a_1,\ldots,a_i,\ldots,a_{i+n-1},\ldots, da_j,\ldots, a_{m+n-1}),
\end{multline*}
\begin{multline*}
p(a_1,\ldots,dq(a_i,\ldots,a_{i+n-1}),\ldots, a_{m+n-1})=\\
\sum_{j=i}^{i+n-1}(p\circ_i q)(a_1,\ldots,a_i,\ldots,da_j,\ldots,a_{i+n-1},\ldots, a_{m+n-1}).
\end{multline*}
In these formulas, we rewrite the elements $p(a_1,\ldots,da_i,\ldots,a_m)$ for which one of $a_1$, \ldots, $a_m$ is computed using the $\calP$-algebra structure. However, according to \cite[Sec.~10.3]{Fr2009}, one may also view $q$ in these formulas as the right $\calP$-module action on $p(a_1,\ldots,da_i,\ldots,a_m)$, and observe that, similarly to the universal multiplicative enveloping algebras, there exists a universal right $\calP$-module of K\"ahler differentials $\Omega^1_\calP$ such that
 \[
\Omega^1_\calP(L)\cong \Omega^1_\calP\circ_{\calP} L.
 \]
It is probably prudent to warn the reader that the standard notation is very misleading here: $\Omega^1_\calP(L)$ looks like the evaluation of the Schur functor $\Omega^1_\calP$ on a vector space $L$, given by the composite product $\Omega^1_\calP\circ_{\mathbb{I}} L$, while in reality it is equal to the \emph{relative} composite product $\Omega^1_\calP\circ_{\calP} L$.

We shall say that a PBW theorem holds for the modules of K\"ahler differentials if there exists a Schur functor~$\calZ$ such that the underlying object of $\Omega^1_\calP(L)$ of any $\calP$-algebra~$L$ is isomorphic to the evaluation $\calZ\circ_{\mathbb{I}} L$ naturally with respect to $\calP$-algebra morphisms. 

\begin{proposition}
Given an operad $\calP$, there is a PBW theorem for the modules of K\"ahler differentials if and only if the right $\calP$-module action on $\Omega^1_\calP$ is free; in this case, the Schur functor $\calZ$ that generates that right module satisfies the above condition:
 \[
\Omega^1_\calP(L)\cong\calZ\circ_{\mathbb{I}}L
 \] 
naturally with respect to $\calP$-algebra morphisms.
\end{proposition}

\begin{proof}
This is analogous to \cite[Th.~3.1]{DT2019}. Indeed, if there is a right module isomorphism $\Omega^1_\calP\cong\calZ\circ\calP$, we have
 \[
\Omega^1_\calP(L)\cong \Omega^1_\calP\circ_{\calP} L\cong(\calZ\circ\calP)\circ_{\calP} L\cong \calZ\circ_{\mathbb{I}}L
 \]
naturally with respect to $\calP$-algebra morphisms. On the other hand, if there is a PBW theorem for the modules of K\"ahler differentials, we shall examine the isomorphism $\calZ\circ_{\mathbb{I}}L\cong\Omega^1_\calP(L)$ for a free algebra $L=\calP\circ_{\mathbb{I}}V$, obtaining
 \[
\calZ\circ_{\mathbb{I}}(\calP\circ_{\mathbb{I}}V)\cong \Omega^1_\calP(\calP\circ_{\mathbb{I}}V)\cong \Omega^1_\calP\circ_{\calP}(\calP\circ_{\mathbb{I}}V)\cong \Omega^1_\calP\circ_{\mathbb{I}} V,
 \]
and it remains to note that this is natural in $V$ (since any map from generators of a free algebra induces an algebra morphism), so we have $\calZ\circ\calP\cong \Omega^1_\calP$, as required.
\end{proof}

We note that the PBW property for the modules of K\"ahler differentials is discussed in \cite{MR2775896} where it is erroneously claimed that it follows from the PBW property for universal multiplicative enveloping algebras; this fails in many cases, starting from the operad of commutative associative algebras.

\section{The universal enveloping pre-Lie algebra of a Lie algebra}

\subsection{The Schur functor}
In this section, we prove the first main result of this paper, the functorial version of the Poincar\'e--Birkhoff--Witt theorem for pre-Lie algebras \cite[Th.~2]{Se1994} which gives a precise description of the underlying vector space of the universal enveloping algebra via a combinatorially defined Schur functor.

To achieve that goal, we recall that according to \cite{CL2001
}, the underlying Schur functor of the operad $\PreLie$ is the linear species of rooted trees $\RT$, and the operad structure is defined in a simple combinatorial way: if $T_1\in\RT(I)$, and $T_2\in\RT(J)$, then for $i\in I$, the element $T_1\circ_i T_2$ is given by
 \[
T_1\circ_i T_2=\sum_{f\colon \mathrm{in}(T_1,i)\to \mathrm{vert}(T_2)} T_1\circ_i^f T_2 .
 \] 
Here $\mathrm{in}(T_1,i)$ is the set of incoming edges of the vertex $i$ in $T_1$ and $\mathrm{vert}(T_2)$ is the set of all vertices of $T_2$; the tree $T_1\circ_i^f T_2$ is obtained by replacing the vertex $i$ of the tree $T_1$ by the tree $T_2$, and grafting the subtree corresponding to the input $v$ of $i$ at the vertex $f(v)$ of $T_2$. For example, we have
 \[
\vcenter{
\xymatrix@M=3pt@R=5pt@C=5pt{
*+[o][F-]{1}\ar@{-}[dr] &&  *+[o][F-]{3}\ar@{-}[dl] \\
& *+[o][F-]{2} & 
}}
\ \circ_2 \ 
\vcenter{
\xymatrix@M=3pt@R=5pt@C=5pt{
*+[o][F-]{a}\ar@{-}[d]  \\
*+[o][F-]{c} 
}}=
\vcenter{
\xymatrix@M=3pt@R=5pt@C=5pt{
*+[o][F-]{1}\ar@{-}[dr] &*+[o][F-]{a}\ar@{-}[d] & *+[o][F-]{3}\ar@{-}[dl]\\
& *+[o][F-]{c} & 
}}+
\vcenter{
\xymatrix@M=3pt@R=5pt@C=5pt{
*+[o][F-]{1}\ar@{-}[dr] & & *+[o][F-]{3}\ar@{-}[dl]\\
 &  *+[o][F-]{a}\ar@{-}[d]& \\
& *+[o][F-]{c} & 
}}+
\vcenter{
\xymatrix@M=3pt@R=5pt@C=5pt{
  & *+[o][F-]{3}\ar@{-}[d]\\
*+[o][F-]{1}\ar@{-}[dr] &  *+[o][F-]{a}\ar@{-}[d]& \\
& *+[o][F-]{c} & 
}}+
\vcenter{
\xymatrix@M=3pt@R=5pt@C=5pt{
*+[o][F-]{1}\ar@{-}[d] & & \\
  *+[o][F-]{a}\ar@{-}[d]& *+[o][F-]{3}\ar@{-}[dl]\\
 *+[o][F-]{c} & 
}}\ .
 \]

There exists a morphism of operads $\phi\colon\Lie\to\PreLie$ defined on the only generator of $\Lie$ as 
 \[
\phi([a_1,a_2])= 
\vcenter{
\xymatrix@M=3pt@R=5pt@C=5pt{
*+[o][F-]{2}\ar@{-}[d]  \\
*+[o][F-]{1} 
}}
-
\vcenter{
\xymatrix@M=3pt@R=5pt@C=5pt{
*+[o][F-]{1}\ar@{-}[d]  \\
*+[o][F-]{2} 
}} .
 \]
The universal pre-Lie algebra of a Lie algebra $\mathfrak{g}$, denoted $U^{\PreLie}(\mathfrak{g})$, is nothing but $\phi_!(\mathfrak{g})$.

\begin{theorem}\label{th:LiePreLie}
Let $\mathfrak{g}$ be a Lie algebra. There is a vector space isomorphism
 \[
U^{\PreLie}(\mathfrak{g})\cong \RT_{\ne 1}(S(\mathfrak{g})),
 \]
where $\RT_{\ne 1}$ is the linear species of rooted trees for which no vertex has exactly one child; moreover, these isomorphisms can be chosen in a way that is natural with respect to Lie algebra morphisms. 
\end{theorem}

\begin{proof}
We shall use the following result (the argument below is a slight elaboration of that of \cite[Sec.~17.1.3]{Fr2009}).

\begin{lemma}\label{lm:freeKoszul}
Let $\calP$ be a Koszul operad equipped with a standard weight grading for which generators are of degree one, and let $\calM$ be a weight graded right $\calP$-module. 
There exists a differential making $\calM\circ\calP^{\ac}$ into a chain complex, such that the module $\calM$ is free if and only if the homology of that complex is concentrated in degree zero; whenever that is the case, the homology represents the Schur functor that freely generates $\calM$ as a right $\calP$-module. 
\end{lemma}

\begin{proof}
Let us consider the derived composite product $\calM\circ^\mathbb{L}_\calP\kk$. According to results of \cite[Sec.~15.2]{Fr2009}, we may compute it using any cofibrant replacement of $\calM$ in the model category of right $\calP$-modules or, alternatively, any cofibrant replacement of $\kk$ in the semi-model category of left $\calP$-modules. If we take the minimal free resolution $\calX_\bullet\circ\calP$ of $\calM$ as a cofibrant replacement, we find that the homological degree $p$ part of $\calM\circ^\mathbb{L}_\calP\kk$ is isomorphic to $\calX_p$, since the composite product with $\kk$ kills the minimal differential. On the other hand, if we take the left Koszul complex $\calP\circ_\kappa\calP^{\ac}=(\calP\circ\calP^{\ac},d_\kappa)$ \cite[Sec.~7.4]{LV2012} as a cofibrant replacement of $\kk$, we see that we can compute the derived composite product using the differential $\id\circ_{\calP} d_\kappa$ on $\calM\circ_{\calP}(\calP\circ\calP^{\ac})\cong \calM\circ\calP^{\ac}$. Since a right $\calP$-module is free if and only if its minimal free resolution coincides with itself, the statement follows.
\end{proof}

In the case of the operad $\calP=\Lie$, this complex has a particularly simple description. Indeed, since the Koszul dual operad of $\Lie$ is $\Com$, the Koszul dual cooperad $\Lie^{\ac}$ is isomorphic to the operadic desuspension of the linear dual $\Com^\vee$; thus, for each $n$ the vector space $\Lie^{\ac}(n)$ is the one-dimensional sign representation of symmetric group placed in homological degree $n-1$. We shall denote the only basis element of that vector space by $e_1\wedge e_2\wedge\cdots\wedge e_n$. The differential $d_\kappa$ of the left Koszul complex $\Lie\circ_\kappa\Lie^{\ac}$ has the following description: it is the unique derivation of left $\Lie$-modules sending each element 
$\id\circ(e_1\wedge e_2\wedge\cdots\wedge e_n)$ to
 \[
\sum_{I\sqcup J=\{1,\ldots,n\}}\epsilon_{I,J} [-,-]\circ (e_I\otimes e_J),
 \]  
where $e_I$ and $e_J$ are the basis elements of $\Lie^{\ac}(I)$ and $\Lie^{\ac}(J)$ respectively, and $\epsilon_{I,J}$ is the Koszul sign, which in this case is simply the coefficient of proportionality between $e_I\wedge e_J$ and $e_1\wedge e_2\wedge\cdots\wedge e_n$. 

We wish to apply this construction to $\calM=\PreLie$. Since the underlying Schur functor of the operad $\PreLie$ is the linear species of rooted trees, each component $\PreLie\circ\Lie^{\ac}(n)$ has a basis of rooted trees whose vertices are decorated by elements $e_I$, so that the subsets $I$ form a partition of $\{1,\ldots,n\}$. Since the morphism from $\Lie$ to $\PreLie$ sends the bracket to the difference of two rooted trees of arity two, the description of the composition in the operad $\PreLie$ implies that the differential $\id\circ_{\Lie} d_\kappa$ on $\PreLie\circ\Lie^{\ac}$ is reminiscent of the usual graph complex differential \cite{MR3348138}: it is equal to the sum of all possible ways to split a vertex of a tree into two vertices connected with an edge and to distribute everything (the edges adjacent to that vertex and the wedge factors of its $e_I$ label) in all possible ways between the two new vertices; the sign of this term is the same Koszul sign as the one given in the formula for $d_\kappa$ above. For example, we have
 \[
\id\circ_{\Lie} d_\kappa\left(
\vcenter{
\xymatrix@M=3pt@R=5pt@C=5pt{
*+[o][F-]{3}\ar@{-}[d]  \\
*+[o][F-]{\vphantom{\frac{\frac{100}{100}}{\frac{100}{100}}}e_1\wedge e_2} 
}}\right)=
\vcenter{
\xymatrix@M=3pt@R=5pt@C=5pt{
*+[o][F-]{3}\ar@{-}[d]  \\
*+[o][F-]{\star} 
}}\circ_\star
\left(
\vcenter{
\xymatrix@M=3pt@R=5pt@C=5pt{
*+[o][F-]{2}\ar@{-}[d]  \\
*+[o][F-]{1} 
}}-
\vcenter{
\xymatrix@M=3pt@R=5pt@C=5pt{
*+[o][F-]{1}\ar@{-}[d]  \\
*+[o][F-]{2} 
}}
\right)=
\vcenter{
\xymatrix@M=3pt@R=5pt@C=5pt{
*+[o][F-]{3}\ar@{-}[d] &  \\
  *+[o][F-]{2}\ar@{-}[d]& \\
 *+[o][F-]{1} & 
}}
+
\vcenter{
\xymatrix@M=3pt@R=5pt@C=5pt{
*+[o][F-]{3}\ar@{-}[dr] & & *+[o][F-]{2}\ar@{-}[dl]\\
 &  *+[o][F-]{1}& \\
}}
-
\vcenter{
\xymatrix@M=3pt@R=5pt@C=5pt{
*+[o][F-]{3}\ar@{-}[dr] & & *+[o][F-]{1}\ar@{-}[dl]\\
 &  *+[o][F-]{2}& \\
}}
-
\vcenter{
\xymatrix@M=3pt@R=5pt@C=5pt{
*+[o][F-]{3}\ar@{-}[d] &  \\
*+[o][F-]{1}\ar@{-}[d]& \\
*+[o][F-]{2} & 
}} ,
 \]
where the right hand side is an element of $\PreLie\cong\PreLie\circ\id\subset \PreLie\circ\Lie^{\ac}$. 

To proceed, we define the \emph{frame} of a rooted tree as the longest path starting from the root and consisting of vertices that have exactly one child (the last point of the frame is the first vertex with at least two children or a leaf). Clearly, the frame is unique, since the moment there is ambiguity is the moment one has to stop; for instance, if the root has more than one child, the frame simply coincides with the root. If we examine the above description of the differential, we see that for a basis element $u\in\PreLie\circ\Lie^{\ac}$ whose underlying rooted tree has a frame of length $d$, the element $\id\circ_{\Lie} d_\kappa(u)$ is a sum of basis elements whose frame is of length at most $d+1$; moreover, the only terms with the frame of length $d+1$ are those where we only split vertices of the frame, and in each such case we reconnect all the children to the new vertex that is further from the root. 

Let us define for each $k$ the graded vector space $F^p(\PreLie\circ_\kappa\Lie^{\ac})(n)$ as the span of all decorated trees for which the difference of the length of the frame and the homological degree (that comes from the labels of vertices; as indicated above, the degree of $e_I$ is $|I|-1$) does not exceed $p$. These spaces form an increasing filtration of $\PreLie\circ_\kappa\Lie^{\ac}(n)$; as we just saw, the space $F^p(\PreLie\circ_\kappa\Lie^{\ac})(n)$ is preserved by the differential. For each $n$, the complex $\PreLie\circ_\kappa\Lie^{\ac}(n)$ is finite-dimensional, therefore the spectral sequence associated to our filtration converges to the homology of this complex. 

Let us examine the first page of our spectral sequence. The differential of the associated graded complex is much simpler, since the only terms from $\id\circ_{\Lie} d_\kappa$ that are kept are those where the length of the frame increases by $1$. The frame of a rooted tree can be identified with an element of $\Ass(I)$, where $I$ is the set of vertices of the frame, therefore, the decorated frame of a decorated rooted tree from $\PreLie\circ_\kappa\Lie^{\ac}$ can be identified with an element of $\Ass\circ\Lie^{\ac}(J)$, where $J$ is the union of all subsets $I$ for all decorations $e_I$ of vertices of the frame. Since the associated graded differential only modifies the frame and does not change the rest of the tree, the associated graded chain complex is isomorphic to the direct sum of complexes with the fixed set $J$ decorating the frame, each such complex being a tensor product of the complex $\Ass\circ\Lie^{\ac}(J)$ with a certain graded vector space (a chain complex with zero differential). Moreover, the differential of the complex $\Ass\circ\Lie^{\ac}(J)$ is precisely $\id\circ_{\Lie} d_\kappa$, if we identify $\Ass\circ\Lie^{\ac}$ with $\Ass\circ_{\Lie}(\Lie\circ\Lie^{\ac})$, so the homology of this complex is simply $\Com(J)$ concentrated in degree zero, since the operad $\Ass$ is free as a right $\Lie$-module, and the generators may be identified with the underlying Schur functor of~$\Com$. Thus, computing the homology of the first page of the spectral sequence amounts to collapsing the frame of each tree to a root labelled by an element of $\Com$ concentrated in homological degree zero. Each child of this root is a root of a subtree from $\PreLie\circ_\kappa\Lie^{\ac}(n)$, and each further differential of the spectral sequence is acting on those trees independently, so one may complete the proof by induction on arity, showing that the homology is concentrated in degree zero and is isomorphic to the Schur functor $\RT_{\ne 1}\circ\Com$. As a consequence, the underlying vector space of $U{\PreLie}(\mathfrak{g})$ is naturally isomorphic to
 \[
\PreLie\circ_{\Lie}(\mathfrak{g})\cong (\RT_{\ne 1}\circ\Com\circ\Lie)\circ_{\Lie}(\mathfrak{g})
\cong(\RT_{\ne 1}\circ\Com)(\mathfrak{g})\cong\RT_{\ne 1}(S(\mathfrak{g})) ,
 \]
as required. 
\end{proof}

\subsection{Applications}

In this section, we record several consequences of Theorem~\ref{th:LiePreLie}, and new directions prompted by it. 

\subsubsection{The operad of \texorpdfstring{$F$}{F}-manifold algebras}

Recall \cite{Do2019} that the operad $\FM$ of $F$-manifold algebras is generated by a symmetric binary operation $-\circ-$ and a skew-symmetric binary operation $[-,-]$ satisfying the associativity relation and the Jacobi identity 
\begin{gather*}
(a_1\circ a_2)\circ a_3=a_1\circ(a_2\circ a_3),\\
[[a_1,a_2], a_3]+[[a_2,a_3], a_1]+[[a_3,a_1],a_2]=0,
\end{gather*}
and related to each other by the Hertling--Manin relation \cite{MR1680372}
\begin{multline*}
[a_1\circ a_2,a_3\circ a_4]=
[a_1\circ a_2, a_3]\circ a_4+[a_1\circ a_2, a_4]\circ a_3+a_1\circ [a_2, a_3\circ a_4]+a_2\circ [a_1, a_3\circ a_4]-\\
(a_1\circ a_3)\circ[a_2,a_4]-(a_2\circ a_3)\circ[a_1,a_4]-(a_2\circ a_4)\circ[a_1,a_3]-(a_1\circ a_4)\circ[a_2,a_3] .
\end{multline*}

\begin{corollary}
The operad $\FM$ is a free right $\Lie$-module; the Schur functor of generators is isomorphic to $\RT_{\ne 1}\circ\Com$. 
\end{corollary}

\begin{proof}
The main theorem asserts that the associated graded of the operad $\FM$ with respect to the filtration defined by powers of the ideal generated by $\Lie$ is isomorphic to the operad $\PreLie$. This means that the underlying Schur functors of $\FM$ and $\PreLie$ are isomorphic. Moreover, the proof of the main theorem of \cite{Do2019} shows that shuffle trees that are not divisible by the trees corresponding to the operations
 \[
[a_1\circ a_2,a_3\circ a_4], \quad
[a_1\circ a_3,a_2\circ a_4], \quad
[a_1\circ a_4,a_2\circ a_3]  
 \]  
form a basis of the operad $\FM$. The rest of the proof proceeds similarly to \cite[Th.~4(2)]{Do2013}. First, we see that the shuffle operad $\FM^f$ associated to $\FM$ is a free right module over the shuffle operad $\Lie^f$ associated to $\Lie$: for the generators of that module, one may take all shuffle trees from the basis above satisfying one extra property: no vertex labelled $[-,-]$ has both children being leaves. For the ordered species $\calY$ spanned by such shuffle trees, the natural map from the shuffle composition of $\calY$ and $\Lie^f$ to $\FM^f$ is an isomorphism. Second, we note that freeness of the right module can be encoded by vanishing of homology groups of the appropriate bar complex, and therefore can be checked on the level of associated shuffle operads. 
\end{proof}

We note that our calculation implies that the underlying Schur functor of the operad $\FM$ is $\RT_{\ne 1}\circ\Poisson$. Given that the Hertling-Manin relation may be viewed \cite{MR1680372} as a weakened version of the derivation rule in a Poisson algebra, it would be interesting to give a more direct proof of this result.

\subsubsection{Permutative homology of commutative algebras}

Suppose that $f\colon\calP\to\calQ$ is a map of Koszul operads. Using the Koszul dual map $f^{\ac}\colon\calP^{\ac}\to\calQ^{\ac}$, one makes $\calP^{\ac}$ a right $\calQ^{\ac}$-comodule. In \cite[Th.~3.7(II)]{griffin2014operadic}, Griffin establishes that if that right comodule is cofree with $\calX$ as the Schur functor of cogenerators, then for each $\calQ$-algebra $A$, one has an isomorphism of complexes 
 \[
\mathsf{B}_{\calP}(f^*A)\cong \calX\circ_{\mathbb{I}}\mathsf{B}_{\calQ}(A).
 \]
Here $f^*A$ is the $\calP$-algebra obtained from the algebra $A$ by the pullback of structure via~$f$. We may apply this result in the case where $\calP=\Perm$ is the operad of permutative algebras \cite{MR1860996} (which is the Koszul dual of the operad $\PreLie$), $\calQ=\Com$, and $f\colon\Perm\to\Com$ is the obvious projection (this map is discussed in \cite[Example~3.8]{griffin2014operadic}, but without the knowledge of the Schur functor of cogenerators it is not possible to write down a specific formula). 

\begin{corollary}
Let $A$ be a commutative associative algebra, and let $\mathsf{B}_{\Perm}(A)$ be bar complex of that algebra considered as a permutative algebra. There is an isomorphism of complexes
 \[
\mathsf{B}_{\Perm}(A)\cong\RT_{\ne 1}(\Com(\mathsf{B}_{\Com}(A)) .
 \]
\end{corollary}

The K\"unneth formula for the composite product of Schur functors implies the same result on the level of homology. This ultimately leads to a question of a definition of a homotopy invariant notion of the universal enveloping pre-Lie algebra (of a Lie algebra up to homotopy). A general framework for addressing questions like that is discussed in a paper of Khoroshkin and Tamaroff \cite{KT}; it would be interesting to apply their methods in this particular case.

One may also look at the permutative (co)homology with coefficients of a commutative associative algebra $A$; an interesting example is given by the deformation complex of the algebra $A$ in the category of permutative algebras. According to \cite[Th.~4.6]{griffin2014operadic}, for questions like that it is reasonable to use information on the structure of the operad $\PreLie$ as a $\Lie$-bimodule. That seems a much harder question; the corresponding bimodule is certainly not free, and not much is known about it. It would be interesting to construct its minimal resolution by free resolution of $\PreLie$ by free $\Lie$-bimodules. We hope to address this question elsewhere.  

\subsubsection{A conjectural good triple \texorpdfstring{$(\calX^c,\PreLie,\Lie)$}{XPL}} 

In the case of universal enveloping associative algebras of Lie algebras, the celebrated Milnor--Moore theorem \cite{MR174052} states that a connected graded cocommutative bialgebra is the universal enveloping associative algebra of its Lie algebra of primitive elements. Conceptually, that result is a manifestation of the general theory of good triples of operads \cite{Loday2008}: in modern terms, it follows from the fact that there exists a good triple 
 \[
(\Com^c,\Ass,\Lie).
 \] Thus, a direct generalisation of that criterion to the case of universal enveloping pre-Lie algebras would require to find a good triples of operads 
 \[
(\calX^c,\PreLie,\Lie),
 \]
so that each universal enveloping algebra $U^{\PreLie}(\mathfrak{g})$ has a coalgebra structure over the cooperad $\calX^c$ for which $\mathfrak{g}$ is a space of primitive elements. To the best of our knowledge, the question of existence of such good triple remains open; it was asked by Loday about fifteen years ago and is recorded as \cite[Problem~7.3]{Markl2007}. Theorem \ref{th:LiePreLie} offers a useful hint for this question, implying that that the underlying species of the cooperad $\calX$ must be $\RT_{\ne 1}\circ\Com$.

\section{Universal multiplicative envelopes and K\"ahler differentials}

In this section, we prove two other functorial Poincar\'e--Birkhoff--Witt theorems in the context of pre-Lie algebras. The proofs of the two results are quite similar; both of them use combinatorics related to that of the vertebrate species of Joyal \cite{MR633783} and the operad $\NAP$ of nonassociative permutative algebras. Algebras over the operad $\NAP$ have a binary product $\triangleleft$ satisfying the identity $(a_1\triangleleft a_2)\triangleleft a_3=(a_1\triangleleft a_3)\triangleleft a_2$. It is known \cite{Li2006} that the underlying Schur functor of this operad is also the linearisation of the species $\partial(\RT)$ of rooted trees, equipped with the combinatorial substitution given by the single tree $T_1\circ_i^f T_2$ for the unique function $f$ sending every vertex of $\mathrm{in}(T_1,i)$ to the root of $T_2$. 
One may view $\NAP$ as a degeneration of the operad $\PreLie$ in the following way \cite{DF}. If we consider, for each rooted tree, the partial order on its set of vertices induced by it, then for any $i\in I$, and for any trees $T\in\RT(I)$ and $S\in\RT(J)$, the composition $T\circ_i S$ in the operad $\PreLie$ is the sum of all trees whose partial order on $I\circ_i J$ refines the order obtained from the orders on $I$ and on $J$ by identification of the root vertex of $S$ with $i$ and whose restrictions to $I$ and to $J$ coincide with the partial orders prescribed by $S$ and by $T$ respectively. 

\subsection{The universal multiplicative enveloping algebra of a pre-Lie algebra}

The result of this section is a strengthening of the PBW theorem for universal multiplicative envelopes of pre-Lie algebras \cite[Th.~1]{KU2004}. A possibility of such a result is indicated in \cite[Th.~5.4]{Kh2019}; however, the proof of that paper utilises the shuffle operad criterion of freeness \cite[Th.~4]{Do2013}, and therefore no filtrations of algebras can be obtained from filtrations of operads, and no direct conclusion about Schur functors can be made. 

\begin{theorem}\label{th:AsPreLie}
Let $L$ be a pre-Lie algebra. There is a vector space isomorphism
 \[
U_{\PreLie}(L)\cong T(L)\otimes S(L).
 \]
Moreover, these isomorphisms can be chosen in a way that is natural with respect to pre-Lie algebra morphisms. 
\end{theorem}

\begin{proof}
To analyse the associative universal enveloping algebra $U_{\PreLie}(V)$, one has to consider the right $\PreLie$-module $\partial(\PreLie)$. We note that the corresponding Schur functor is the linearisation of the species $\partial(\RT)$ of rooted trees where all of the vertices but one are labelled. 

To show that the right $\PreLie$-module $\partial(\PreLie)$ is free, we shall first prove the same for the operad $\NAP$. 

\begin{lemma}
The species derivative $\partial(\NAP)$ is a free right $\NAP$-module; specifically, we have a right module isomorphism
 \[
\partial(\NAP)\cong(\uAss\otimes\uCom)\circ\NAP . 
 \]
\end{lemma}

\begin{proof}
This statement is almost obvious. Namely, we realise the Schur functor $\uAss\otimes\uCom$ inside $\partial(\NAP)$ as the linearization of the species of rooted trees for which each vertex on the unique path from the root to the unlabelled vertex different from the starting point and the end point of the path has exactly one child, and each child of the unlabelled vertex has no children of its own. Thus, the vertices on the path from the root to the unlabelled vertex represent $\uAss$, the children of the unlabelled vertex represent $\uCom$, and the way the tree is assembled of those corresponds to computing the tensor product. Substitution in the operad $\NAP$ amounts to grafting trees at all these vertices, and it is true by direct inspection that each rooted tree is obtained from  $\uAss\otimes\uCom$ by the right $\NAP$-module action in the unique way (by forming the subgraph consisting of the path from the root to the unlabelled vertex and substituting the ``missing parts'' into the labelled vertices of this subgraph).  
\end{proof}

Let us show now that we have a right module isomorphism
 \[
\partial(\PreLie)\cong(\uAss\otimes\uCom)\circ\PreLie . 
 \]
Indeed, let us consider the same realization of the Schur functor $\uAss\otimes\uCom$ inside $\partial(\PreLie)$ via the species of rooted trees for which each vertex on the unique path from the root to the unlabelled vertex different from the starting point and the end point of the path has exactly one child, and each child of the unlabelled vertex has no children of its own. Since $\NAP$ is a degeneration of the operad $\PreLie$, an inductive argument handling the lower terms immediately shows that this Schur functor generates $\partial(\PreLie)$ as a right $\PreLie$-module. Similarly, any right $\PreLie$-module dependency has the ``leading term'' corresponding to the operadic compositions in $\NAP$, so a nontrivial $\PreLie$-module dependency necessarily implies a nontrivial $\NAP$-module dependency, which is a contradiction. 

It remains to notice that
\begin{multline*}
U_{\PreLie}(L)\cong\partial(\PreLie)\circ_{\PreLie}L\cong\\ ((\uAss\otimes\uCom)\circ\PreLie)\circ_{\PreLie}L\cong
 (\uAss\otimes\uCom)\circ_{\mathbb{I}}L\cong T(L)\otimes S(L) , 
\end{multline*}
the isomorphism being natural in~$L$, as required. 
\end{proof}

\subsection{The module of K\"ahler differentials of a pre-Lie algebra}

The result of this section is a strengthening of the PBW theorem for modules of K\"ahler differentials of pre-Lie algebras~\cite[Cor.~4]{KU2004}. 

\begin{theorem}\label{th:KahlerPreLie}
Let $L$ be a pre-Lie algebra. There is a vector space isomorphism
 \[
\Omega^1_{\PreLie}(L)\cong \overline{T}(L),
 \]
where $\overline{T}(L)$ is the augmentation ideal of the tensor algebra. Moreover, these isomorphisms can be chosen in a way that is natural with respect to pre-Lie algebra morphisms. 
\end{theorem}

\begin{proof}
As in Theorem \ref{th:AsPreLie}, we shall first establish an appropriate freeness result for the operad $\NAP$. 

\begin{lemma}
The right $\NAP$-module $\Omega^1_{\NAP}$ is free; specifically, we have a right module isomorphism
 \[
\Omega^1_{\NAP}\cong\Ass\circ\NAP . 
 \]
\end{lemma}

\begin{proof}
By definition, the Schur functor $\Omega^1_{\NAP}$ is the linearisation of the species of vertebrates \cite{MR633783}, which encodes labelled trees with two special vertices, the head and the tail. (The head of the vertebrate corresponds to the root, while the tail indicates the vertex where $d$ is applied.) Let us consider, for each vertebrate, its \emph{vertebral column}, defined as the unique path from the head to the tail. Note that the species of vertebral columns is the same as the species of nonempty linear orders, whose linearisation is the underlying Schur functor of the operad $\Ass$. We shall show that that latter Schur functor freely generates $\Omega^1_{\NAP}$ as a right $\NAP$-module. By definition of the module of K\"ahler differentials, operadic compositions at all vertices except for the tail of the vertebral column are the $\NAP$-compositions, and the operadic composition at the tail is slightly more complicated. Specifically, that latter composition corresponds to performing the operadic composition, and then summing up all possible ways to insert $d$ at the vertices of the inserted tree. In other words, if $T_1$ is a vertebrate, $T_2$ is a rooted tree, and $t$ is the tail of $T_1$, one computes the right $\NAP$-action $T_1\circ_t T_2$ by composing $T_1$ and $T_2$ in the operad $\NAP$ and then summing up all ways to make that rooted tree into a vertebrate by choosing a vertex coming from $T_2$. 

We may consider the decreasing filtration of the module $\Omega^1_{\NAP}$ by the length of the vertebral column: $F^p\Omega^1_{\NAP}$ consists of vertebrates with the vertebral column of length at least $p$. We have just seen that the right module action has the terms where the length of the vertebral column is unchanged and some terms where it increases. The associated graded module is manifestly isomorphic to $\Ass\circ\NAP$: each vertebrate is obtained from  $\Ass$ by the right $\NAP$-module action in the unique way (by considering the vertebral column and substituting the ``missing parts'' into the its vertices). We shall now use Lemma \ref{lm:freeKoszul}. For that, we notice that the filtration of the right module $\Omega^1_{\NAP}$ induces a filtration of the complex $\Omega^1_{\NAP}\circ_\kappa\NAP^{\ac}$ for which vanishing of homology in positive degrees is equivalent to freeness (since the operad $\NAP$ is well known to be Koszul). For each $n$, the complex $\Omega^1_{\NAP}\circ_\kappa\NAP^{\ac}(n)$ is finite-dimensional, therefore the spectral sequence associated to the filtration converges to the homology of this complex. Since the associated graded module is free, the homology of the first page of the spectral sequence is concentrated in degree zero, so the same is true for the complex $\Omega^1_{\NAP}\circ_\kappa\NAP^{\ac}$, meaning that the right $\NAP$-module $\Omega^1_{\NAP}$ is free.  
\end{proof}

It follows by an argument analogous to that of in Theorem \ref{th:AsPreLie} that the right $\PreLie$-module $\Omega^1_{\PreLie}$ is free, with the same generators. Finally, we notice that
 \[
\Omega^1_{\PreLie}(L)\cong\Omega^1_{\PreLie}\circ_{\PreLie}L\cong (\Ass\circ\PreLie)\circ_{\PreLie}L\cong
 \Ass\circ_{\mathbb{I}}L\cong \overline{T}(L), 
 \]
the isomorphism being natural in~$L$, as required. 
\end{proof}

\subsection{Application to pre-Lie algebra cohomology}

Let us discuss the meaning of our results in the context of computing cohomology of pre-Lie algebras with coefficients in their modules. We believe that the first explicit definition of cohomology of an algebra $L$ over a Koszul operad $\calP$ generated by binary operations with coefficients in a module $E$ was given in papers of Kimura and Voronov \cite{MR1363061} and Fox and Markl \cite{MR1436921} (their definition works for any quadratic operad, but is meaningful from the homotopy theory point of view only in the Koszul case). This cohomology theory was further studied by 
Balavoine \cite{MR1642086} who in particular gave a more concrete formula for the differential. In modern terms, this cohomology of a $\calP$-algebra $L$ with coefficients in a module $E$ is computed by a complex 
 \[
C^{\bullet}_{\calP}(L,E):=(\Hom(\calP^{\ac}\circ_{\mathbb{I}}L,E),d),
 \]
where the differential $d$ is defined by including that complex in a bigger complex, the operadic deformation complex of the square-free extension algebra $L\oplus E$. In particular, this may be applied to the Koszul operad $\PreLie$ for which $\PreLie^{\ac}\circ_{\mathbb{I}}L\cong L\otimes\Com^\vee(s L)$.

Another definition, this time specific to the case of pre-Lie algebras, is the \emph{ad hoc} definition of cohomology proposed by Dzhumadildaev \cite{MR1698758}. According to that definition, the cohomology is computed using a particular cochain complex constructed by hand; it is given by $\Hom(L\otimes\Lambda(L),E)$ in positive cohomological degrees and a certain space constructed by hand in cohomological degree zero. 
Here $\Lambda(L)$ is the Grassmann algebra of $L$, whose underlying vector space is naturally identified with $\Com^\vee(s L)$, so up to a degree shift by one, the main part of this cochain complex matches that of Balavoine (and the differentials match as well). In \cite{MR1698758}, this cohomology is referred to as a derived functor in the category of $U_{\PreLie}(L)$-modules, though this assertion is never proved (and is unlikely, given that the ``derived functor'' calculation does not seem to use projective resolutions). 

In general, contrary to the intuition coming from studying cohomology of Lie algebras and of associative algebras, cohomology of algebras over operads is not given by the Ext-functor over the universal multiplicative enveloping algebra. We shall now show that it is true in the case of pre-Lie algebras.

\begin{corollary}
For every pre-Lie algebra $L$ and every $L$-module $E$, we have
 \[
H_{\PreLie}^\bullet(L,E)\cong{\mathrm{Ext}}^\bullet_{U_{\PreLie}(L)}(\Omega^1_{\PreLie}(L),E).
 \] 
\end{corollary}

\begin{proof}
We shall apply the general result of Fresse \cite[Th.~17.3.4]{Fr2009} that gives a criterion for the isomorphism 
 \[
H_{\calP}^\bullet(L,E)\cong{\mathrm{Ext}}^\bullet_{U_{\calP}(L)}(\Omega^1_{\calP}(L),E)
 \] 
to hold. This criterion requires that the operad $\calP$ is $\Sigma_*$-cofibrant (which is automatic over a field of zero characteristic) and that $\partial(\calP)$ and $\Omega^1_{\calP}$ are cofibrant as right $\calP$-modules (which is true for free right modules over an operad with zero differential). Since in the proof of Theorem \ref{th:AsPreLie} we established that $\partial(\PreLie)$ is a free right $\PreLie$-module and in the proof of Theorem \ref{th:KahlerPreLie} we established that $\Omega^1_{\PreLie}$ is a free right $\PreLie$-module, the result follows. 
\end{proof}

\section*{Acknowledgements } The authors thank Anton Khoroshkin, Martin Markl and Pedro Tamaroff for several useful comments. Special thanks are due to Frederic Chapoton whose remarks on another project of the first author made us discover Theorem \ref{th:KahlerPreLie}. The substantial revision of the original version of this paper was done during the first author's visit to Centre de Recerca Matem\`atica in Barcelona, and he wishes to thank that institution for hospitality and excellent working conditions. 

\bibliographystyle{alpha} 
\bibliography{LiePreLiePBW.bib} 

\end{document}